\documentclass[11pt]{article}
\usepackage{geometry}
\usepackage{amsmath, amssymb, amsthm, url} 

\usepackage{hyperref}
\usepackage[dvipsnames,svgnames,x11names,hyperref]{xcolor}
\hypersetup{
	colorlinks,
	citecolor=Sepia,
	filecolor=red,
	linkcolor=[RGB]{0,80,200},
	urlcolor=Blue
}

\newtheorem{theorem}{Theorem}

\newtheorem{lemma}[theorem]{Lemma}

\newtheorem{proposition}[theorem]{Proposition}

\theoremstyle{definition}

\newtheorem{remark}[theorem]{Remark}
\newtheorem*{remark*}{Remark}

\renewcommand{\epsilon}{\varepsilon}

\newcommand{\eps}{\varepsilon}

\renewcommand{\leq}{\leqslant}
\renewcommand{\geq}{\geqslant}
\renewcommand{\dim}{n}

\newcommand{\bits}{\{0,1\}^{\dim} }

\newcommand{\A}{\mathcal{A}}

\newcommand{\norm}[1]{\|#1\|}
\newcommand{\ceil}[1]{\lceil #1 \rceil}
\newcommand{\floor}[1]{\lfloor #1 \rfloor}

\newcommand{\er}{\mathsf{e}_{\leq r}}
\newcommand{\Er}{\mathsf{E}_{\leq r}}

\newcommand{\calU}{{\cal U}}

\newcommand{\B}{\mathcal{B}}
\newcommand{\jth}{j^{\mathrm{th}}}
\renewcommand{\b}{\beta}

\newcommand{\HLinkShort}[2]{\hyperref[#2]{#1\ref*{#2}}}
\newcommand{\HLink}[2]{\hyperref[#2]{#1~\ref*{#2}}}
\newcommand{\HLinkPage}[2]{\hyperref[#2]{#1~\ref*{#2}%
		$_\text{p\pageref{#2}}$}}
\newcommand{\HLinkPageOnly}[1]{\hyperref[#1]{Page~\refpage*{#1}%
		$_\text{p\pageref{#1}}$}}

\newcommand{\HLinkSuffix}[3]{\hyperref[#2]{#1\ref*{#2}{#3}}}
\newcommand{\HLinkPageSuffix}[3]{\hyperref[#2]{#1\ref*{#2}%
		#3$_\text{p\pageref{#2}}$}}

\newcommand{\seclab}[1]{\label{section:#1}}
\newcommand{\secref}[1]{\HLink{Section}{section:#1}}

\newcommand{\lemlab}[1]{\label{lemma:#1}}
\newcommand{\lemref}[1]{\HLink{Lemma}{lemma:#1}}%

\newcommand{\proplab}[1]{\label{prop:#1}}

\newcommand{\propref}[1]{\HLink{Proposition}{prop:#1}}

\newcommand{\thmlab}[1]{{\label{theo:#1}}}
\newcommand{\thmref}[1]{\HLink{Theorem}{theo:#1}}

\providecommand{\eqlab}[1]{}%
\renewcommand{\eqlab}[1]{\label{equation:#1}}

\newcommand{\Eqref}[1]{\HLinkSuffix{Eq.~(}{equation:#1}{)}}

\title{Edge Isoperimetric Inequalities for Powers of the Hypercube}

\author{Cyrus Rashtchian\\
\small Dept.~of Computer Science \& Engineering\\[-0.8ex]
\small UC San Diego\\[-0.8ex] 
\small La Jolla, CA, U.S.A.\\
\small\tt \url{crashtchian@eng.ucsd.edu}\\
\and
William Raynaud\\
\small School of Mathematical Sciences\\[-0.8ex]
\small Queen Mary, University of London\\[-0.8ex]
\small London, England, U.K.\\
\small\tt \url{wgraynaud@gmail.com}}

\begin{document}

\maketitle


\begin{abstract}
 For positive integers $n$ and $r$, we let $Q_n^r$ denote the $r$th power of the $n$-dimensional discrete hypercube graph, i.e., the graph with vertex-set $\{0,1\}^n$, where two 0-1 vectors are joined if they are Hamming distance at most $r$ apart. We study edge isoperimetric inequalities for this graph. Harper, Bernstein, Lindsey and Hart proved a best-possible edge isoperimetric inequality for this graph in the case $r=1$. For each $r \geq 2$, we obtain an edge isoperimetric inequality for $Q_n^r$; our inequality is tight up to a constant factor depending only upon $r$. Our techniques also yield an edge isoperimetric inequality for the `Kleitman-West graph' (the graph whose vertices are all the $k$-element subsets of $\{1,2,\ldots,n\}$, where two $k$-element sets have an edge between them if they have symmetric difference of size two); this inequality is sharp up to a factor of $2+o(1)$ for sets of size $\binom{n -s}{k-s}$, where $k=o(n)$ and $s \in \mathbb{N}$.
\end{abstract}

\section{Introduction}
Isoperimetric questions are classical objects of study in mathematics. In general, they ask for the minimum possible `boundary-size'
of a set of a given `size', where the exact meaning of these words varies according to the problem. A classical example of an isoperimetric problem is to minimise the perimeter among all shapes in the plane with unit area. The solution to this problem was `known' to the Ancient Greeks, but the first rigorous proof was given by Weierstrass in a series of lectures in Berlin in the 1870s.

In the last fifty years, there has been a great deal of interest in {\em discrete} isoperimetric inequalities. These deal with the boundaries of sets of vertices in graphs. If $G = (V,E)$ is a graph, and $A \subset V(G)$ is a subset of vertices of $G$, the {\em edge boundary} of $A$ consists of the set of edges of $G$ which join a vertex in $A$ to a vertex in $V(G) \setminus A$; it is denoted by $\partial_{G}(A)$, or by $\partial A$ when the graph $G$ is understood. The {\em edge isoperimetric problem for $G$} asks for a determination of $\min\{|\partial A|:\ A \subset V(G),\ |A|=m\}$, for each integer $m$. 

If $G = (V,E)$ is a graph and $A \subset V(G)$, we write $e_G(A)$ for the number of edges of $G$ induced by $A$, i.e., the number of edges of $G$ that join two vertices in $A$. We remark that if $G$ is a regular graph, then the edge isoperimetric problem for $G$ is equivalent to finding the maximum possible number of edges induced by a set of given size. Indeed, if $G$ is a $d$-regular graph, then
\begin{equation}\label{eq:reg} 2e_G(A) + |\partial A| = d|A|\end{equation}
for all $A \subset V(G)$.

An important example of a discrete isoperimetric problem is the edge isoperimetric problem for the Hamming graph $Q_n$ of the $n$-dimensional hypercube. We define $Q_n$ to be the graph with vertex-set $\{0,1\}^n$, where two 0-1 vectors are adjacent if they differ in exactly one coordinate. This isoperimetric problem has numerous applications, both to other problems in mathematics, and in other areas such as distributed algorithms~\cite{beame2017massively, sarma2013upper}, communication complexity~\cite{harper}, network science~\cite{bezrukov} and game theory~\cite{hart}.

The edge isoperimetric problem for $Q_n$ has been solved by Harper \cite{harper}, Lindsey~\cite{lindsey}, Bernstein~\cite{bernstein} and Hart~\cite{hart}. Let us describe the solution. The {\em binary ordering} on $\{0,1\}^n$ is defined by $x < y$ if and only if $\sum_{i=1}^{n} 2^{i-1} x_i < \sum_{i=1}^{n} 2^{i-1}y_i$. If $m \leq 2^n$, the {\em initial segment of the binary ordering on $\{0,1\}^n$ of size $m$} is simply the subset of $\{0,1\}^n$ consisting of the $m$ smallest elements of $\{0,1\}^n$ with respect to the binary ordering. Note that if $m = 2^d$ for some $d \in \mathbb{N}$, then the initial segment of the binary ordering on $\{0,1\}^n$ of size $m$ is the $d$-dimensional subcube $\{x \in \{0,1\}^n:\ x_i = 0 \ \forall i > d\}$.

Harper, Bernstein, Lindsey and Hart proved the following.
\begin{theorem}[The edge isoperimetric inequality for $Q_n$]
\label{thm:edge-iso}
If $\A \subset \{0,1\}^n$, then $|\partial \A| \geq |\partial \B|$, where $\B \subset \{0,1\}^n$ is the initial segment of the binary ordering of size $|\A|$.
\end{theorem}

\noindent In particular, it follows from Theorem \ref{thm:edge-iso} that the minimum edge-boundary of a set of size $2^{d}$ is attained by a $d$-dimensional subcube, for any $d \in \mathbb{N}$. As another consequence, the above theorem implies that $e_{Q_n}(\A) \leq \frac{1}{2}|\A|\log_2 |\A|$ for all $\A \subset \bits$.

For background on other discrete isoperimetric inequalities, we refer the reader to the surveys of Bezrukov \cite{bezrukov} and of Leader \cite{leader-survey}.

In this paper, we consider the edge isoperimetric problem for {\em powers} of the hypercube. If $r,n \in \mathbb{N}$, we let $Q_n^r$ denote the $r$th power of $Q_n$, that is, the graph with vertex-set $\{0,1\}^n$, where two  distinct 0-1 vectors are joined by an edge if they differ in at most $r$ coordinates. Writing $[n]: = \{1,2,\ldots,n\}$, we may identify $\bits$ with the power-set $\mathcal{P}([\dim])$ via the natural bijection $x \leftrightarrow \{i \in [\dim]:\ x_i=1\}$. By doing so, we may alternatively  view $Q_n^r$ as the graph with vertex-set $\mathcal{P}([\dim])$, where two distinct subsets of $[n]$ are joined if their symmetric difference has size at most $r$. As usual, the {\em Hamming weight} of a vector $x \in \{0,1\}^n$ is its number of 1's; if $x,y \in \{0,1\}^n$, the {\em Hamming distance} between $x$ and $y$ is the number of coordinates on which they differ. Hence, two 0-1 vectors are adjacent in $Q_n^r$ if and only if they are Hamming distance at most $r$ apart. 

Note that $Q_n^r$ is a regular graph, so by (\ref{eq:reg}), the edge isoperimetric problem for $Q_n^r$ is equivalent to finding the maximum number of edges of $Q_n^r$ induced by a set of given size. In other words, it is equivalent to determining
$$D(m,n,r) := \max\{e_{Q_n^r}(\A):\ \A \subset \{0,1\}^n,\ |\A|=m\},$$
i.e.\ the maximum possible number of pairs of vectors at Hamming distance $r$ or less, among a set of $m$ vectors in $\{0,1\}^n$, for each $(m,n,r) \in \mathbb{N}^3$. We remark that, since $Q_n^r$ is regular of degree $\sum_{j=1}^{r} {n \choose j}$, one has the trivial upper bound
\begin{equation}\label{eq:triv} D(m,n,r) \leq \tfrac{1}{2}m \sum_{j=1}^{r} {n \choose j}\qquad \forall m,n,r \in \mathbb{N}.\end{equation}

In the light of Theorem \ref{thm:edge-iso}, which gives a complete answer to the isoperimetric problem for~$Q_n^r$ in the case $r=1$, it is natural to ask whether, for each $n \geq r \geq 2$, there exists an ordering of the vertices of $\{0,1\}^n$ such that initial segments of this ordering minimize the edge-boundary in $Q_n^r$, over all sets of the same size. Unfortunately, this is false even for~$r=2$. Indeed, this is easy to check when $r=2$ and $n = 4$, in which case the optimal isoperimetric sets of size~5 are precisely the Hamming balls of radius~1, whereas an optimal set of size~7 must be a 3-dimensional subcube minus a point, which contains no Hamming ball of radius~1. Hence, the problem for $r\geq 2$ is somewhat harder than in the case $r=1$. Still, as we shall see, reasonably good bounds can be obtained in many cases.

The problem of determining (or bounding) $D(m,n,r)$ was considered by Kahn, Kalai and Linial in \cite{kkl}. For half-sized sets, they solve the problem completely, proving that
\begin{equation} \label{eq:kkl-bound} D(2^{n-1},n,r) = 2^{n-2}\sum_{j=1}^{r} {n-1 \choose j} \qquad\forall r,n \in \mathbb{N}.\end{equation}
(For odd $r$, the extremal sets for (\ref{eq:kkl-bound}) are precisely the $(n-1)$-dimensional subcubes; for even $r$, the set of all vectors of even Hamming weight is also extremal.) Kahn, Kalai and Linial also observe that if $(r/n)\log (2^n/m) = o(1)$, then the trivial upper bound (\ref{eq:triv}) is asymptotically sharp, i.e.\
$$D(m,n,r) = (1-o(1))\tfrac{1}{2} \cdot m \sum_{j=1}^{r} {n \choose j};$$
this can be seen by considering the initial segment of the binary ordering on $\{0,1\}^n$ with size~$m$ --- for example a subcube, if $m$ is a power of 2. Finally, they observe that Kleitman's diametric theorem~\cite{Kleitman1966} implies that if $m$ is `very' small, then the `other' trivial upper bound $D(m,n,r) \leq {m \choose 2}$ is sharp. In particular, for even values of $r$ we know that 
$D(m,n,r) = {m \choose 2}$ if and only if $m \leq \sum_{j=0}^{r/2} {n \choose j}.$
In this case, one may consider an $m$-element subset of a Hamming ball of radius $r/2$, which has diameter at most~$r$.
A similar result for small sets and odd $r$ holds as well.

It is also natural to consider the edge isoperimetric problem for the subgraph of $Q_n^r$ induced by the binary vectors of Hamming weight $k$, or equivalently the graph with vertex-set ${[n] \choose k}$ where two $k$-sets are joined if their symmetric difference has size at most~$r$. In the case $r=2$, this graph is called the `Kleitman-West graph', and the edge isoperimetric problem has been called the `Kleitman-West problem' (see e.g.~\cite{harper-kw}). An elegant conjecture of Kleitman (as to the complete solution of the latter edge isoperimetric problem for all $k$ and all vertex-set sizes) was disproved by Ahlswede and Cai \cite{ahlswede-cai-counter}; only for $k\leq 2$ is a complete solution known \cite{ahlswede-cai, ahlswedekatona}. Related results 
have been obtained by Ahlswede and Katona~\cite{ahlswedekatona} and Das, Gan and Sudakov~\cite{das2016minimum} (Theorem 1.8 in the latter paper implies a solution to the Kleitman-West problem for certain large values of $n$, for each fixed $k$). Harper attempted to resolve the edge isoperimetric problem in this case via a continuous relaxation~\cite{harper-kw}. Unfortunately, Harper's argument works only in certain special cases, and he later demoted his claim to a conjecture~\cite{harperbook}.

\subsection{Our results}
We obtain the following bounds on $D(m,n,r)$. For brevity, we state our theorems in terms of the function $\ell = \ell(m) = 
\min\left\{
\left\lceil 
\frac{2\log m}{\log n - \log \log m} 
\right\rceil,
\floor{\log m}
\right\}$. All logs are base two. Our results are only novel when the minimum for $\ell$ is achieved by the first term. This is case when $m$ and $n$ satisfy ${\frac{2\log m}{\log \dim - \log \log m}} \leq {\log m}$, or in other words, when $m \leq 2^{\dim/4}$. We introduce the $\ell$ notation here since we use it in several places in our proofs.

\begin{theorem}
	\thmlab{even-power-approx}
	Let $m,n,t \in \mathbb{N}$ with $2^t \leq m \leq 2^n$. Then
	$$
	D(m,n,2t) \ \leq\  
	\left(\frac{8e}{t}\right)^{2t} 
	\cdot \left(\dim \cdot \ell  \right)^t 
	\cdot m.
	$$
\end{theorem}
\begin{theorem}
	\thmlab{odd-power-approx}
	Let $m,n,t \in \mathbb{N}$ with $2^t \leq m \leq 2^n$. Then
	$$
	D(m,n,2t+1) \ \leq\  
	\left(\frac{16e}{2t+1}\right)^{2t+1} 
	\cdot \left(\dim \cdot \ell  \right)^t 
	\cdot m \cdot \log m.
	$$
\end{theorem}

The two theorems above are tight up to a constant factor depending on $t$, viz., a factor of $\exp(\Theta(t))$; see below for details. 
In the case $r=2$, we prove a sharper bound (\thmref{distance-two}), which implies a new bound for the Kleitman-West problem (\thmref{kwc}). 
Determining the optimal solution to the isoperimetric problem for all vertex-set-sizes remains a challenging open problem, one which seems beyond the reach of our techniques. As mentioned above, even the restriction to $k$-sets and $r=2$ is open for $k\geq 3$, that is, the Kleitman-West problem remains unsolved. 

\paragraph{Tightness.}
	For fixed $t \in \mathbb{N}$, \thmref{even-power-approx} is sharp up to a factor of $\exp(\Theta(t))$, as can be seen by taking $\A = [\dim]^{(\leq k)}$, i.e., a Hamming ball. In fact, this example motivates our definition of $\ell$ above, capturing the way $e_{Q^{2t}_n}(\A)$ scales as a function of $|\A|$.
	
	\thmref{odd-power-approx} is also sharp up to a factor of $\exp(\Theta(t))$, as can be seen by considering
	$$\A_{k,t} = \{x \subseteq [\dim] \ : \ |x \cap \{k-t+1,\ldots,\dim\}| \leq t\}.$$ 
	When $n \geq k \geq t$, we have 
	$$|\A_{k,t}| = 2^{k-t} \sum_{i=0}^t \binom{n-k+t}{i}.$$
    Denoting $\A = \A_{k,t}$, we sketch the calculations for $k = \Theta(\log n)$ and for fixed $t \in \mathbb{N}$. Note that for this parameter range, $\log |\A| = \Theta_t(\log n) = \Theta_t(k)$, and hence, $\ell = \Theta_t(1)$.
    We claim that there are $\Omega_t(1) n^t |\A| \log |\A|$ pairs with Hamming distance $2t+1$. Given our assumptions on $k,t$, this implies that $D(m,n,2t+1) \geq \Omega_t(1) (n \ell)^t m \log m$ for sets of size $m = 2^{\Theta(\log n)}$. We count pairs $x,y \in \A$ with $|x \Delta y| = 2t+1$ and $|(x\setminus y) \cap \{k-t+1,\ldots, n\}| = |(y\setminus x) \cap \{k-t+1,\ldots, n\}| = t$. There are $\binom{n-k}{t} \binom{n-k+t}{t}$ ways to satisfy this equality, and doing so incurs a Hamming distance of $2t$ restricted to $\{k-t+1,\ldots, n\}$. Then, if $x$ has $j$ ones in positions $\{1,\ldots,k-t\}$, changing one of these ones to a zero, i.e., $|(x\setminus y) \cap [k-t]| =1$, leads to $|x \Delta y| = 2t+1$. Hence, the number of such $\{x,y\}$ pairs is 
    \begin{align*}
    \binom{n-k}{t} \binom{n-k+t}{t}\cdot \sum_{j=0}^{k-t} \binom{k-t}{j} j
    &= 
    \binom{n-k}{t} \binom{n-k+t}{t} \cdot 2^{k-t-1} (k-t)\\
    &= \Omega_t(1) \cdot (n-k)^t |\A| (k-t) \\ 
    &= \Omega_t(1) \cdot n^t |\A| k \\
    &= \Omega_t(1) \cdot n^t |\A| \log |\A|.
    \end{align*} 

\paragraph{Independent Work.} Kirshner and Samorodnitsky \cite{KirshnerSamorodnitsky} independently obtained isoperimetric results similar to those proved in this paper. They use very different methods, and we briefly sketch their results here. For any function $f: \{0,1\}^n \rightarrow \mathbb{R}$ and $p \geq 1$, as usual we define the {\em $p$-norm} of $f$ by $\norm{f}_p = \left(\mathbb{E}_x \left[|f(x)|^p \right]\right)^{1/p},$ where the expectation is over a uniformly random $x \in\{0,1\}^n$. Let $H(\cdot)$ be the binary entropy function (i.e., for $q \in (0,1)$ we let $H(q) := -q \log_2(q) - (1-q) \log_2(1-q)$), and let $\psi(p,t)$ be the function on $[2,\infty) \times [0,1/2]$ defined by $$\psi (p,t) = (p-1) + \log_2\left( (1-\delta)^p + \delta^p\right) - \frac{p}{2} H(t) - pt \log_2(1-2\delta),$$ where $\delta$ is determined by $t = (\frac{1}{2} - \delta) \cdot \frac{(1-\delta)^{p-1} - \delta^{p-1}}{(1-\delta)^p + \delta^p}$. Kirshner and Samorodnitsky show that for $p \geq 2$ and $0 \leq s \leq \frac{n}{2}$, and for a homogeneous polynomial $f$ of degree $s$ on $\{0,1\}^n$ we have $$\frac{\norm{f}_p}{\norm{f}_2} \leq 2^{\psi \left(p,s/n \right) \cdot \frac{n}{p}}.$$ Furthermore, they show that in a well-defined sense this inequality is `nearly tight' if $f$ is the Krawchouk polynomial (the Fourier transform of the characteristic function of a Hamming sphere). Kirshner and Samorodnitsky then show that these results imply for each $0 \leq s \leq n/2$ and $1 \leq r \leq 2s(1-\frac{s}{n})$ that 
\begin{equation}\eqlab{KSUpperBound}
D \left( \sum_{t=0}^s \binom{n}{t},n,r \right) \leq \left( \sum_{t=0}^s \binom{n}{t} \right) \cdot 2^{H \left( \frac{r}{2s}\right)\cdot s + H\left(\frac{r}{2(n-s)}\right)\cdot (n-s)}.
\end{equation}
For odd $r$ this upper bound is tight up to a factor of $O\left(\sqrt{\frac{n-s}{s}} \cdot r\right)$. We can compare this to one of the main theorems of this paper, \thmref{odd-power-approx}, which is tight up to a factor $\exp(\Theta(r))$. For fixed $s$, \thmref{odd-power-approx} is stronger for $r < \frac{1}{2} \log n $ and $n$ sufficiently large. However, the isoperimetric bounds achieved by Kirshner and Samorodnitsky for even $r$, which are tight up to a factor of $O(r)$, improve upon the second main theorem of the paper, \thmref{even-power-approx}, which is only tight up to $\exp (\Theta(r))$. Applying the upper bound \Eqref{KSUpperBound} to the Kleitman-West graph ($r=2$), we see that their result implies a bound that is tight up to a factor of $2e^2 \approx 14.778$. This is weaker than our upper bound for $r=2$ in \thmref{kwc}, which is tight up to $2+o(1)$; see \secref{disttwo}, where \thmref{kwc} follows from \thmref{distance-two}. 

\paragraph{Subsequent Work.} In the time since the submission of this paper, \thmref{even-power-approx} and~\ref{theo:odd-power-approx} have been used to obtain new model counting results~\cite{meel2020sparse}. 
Improved bounds on $D(m,n,r)$ have also been proven for large set sizes, $m = \alpha 2^n$ with $\alpha \in (0,1)$, using probabilistic techniques~\cite{yu2022edge}.
In general, there has been much work on discrete isoperimetric inequalities in the Hamming cube \cite{jiang2022isoperimetric, kahn2020isoperimetric, keevash2018stability, keevash2020stability, przykucki2020vertex} and related studies \cite{eldan2020concentration, kahn2019number}.

\subsection{Notation and Preliminaries}
\seclab{notation}

For subsets $\A \subseteq \{0,1\}^n$, we let $\Er(\A)$ denote the set of edges in the subgraph of $Q_n^r$ induced by vertices in $\A$, and we write $\er(\A): =  |\Er(\A)|$. In this notation, notice that $D(m,n,r) = \max_{\A : |\A|=m} \er(\A)$. Abusing notation slightly, we move freely between $\{0,1\}^n$ and $\mathcal{P}([n])$ via the bijection $x \leftrightarrow \{i \in [n]:\ x_i=1\}$. We say $\A \subseteq \bits$ is a {\em down-set} if ($x\in \A$, $y \subseteq x) \Rightarrow y \in \A$. We say $\A$ is {\em left-compressed} if whenever $1\leq i < j \leq n$ and $x\in \A$ with $x \cap \{i,j\}=\{j\}$, we have $(x \cup \{i\}) \setminus \{j\} \in \A$. 

Standard compression arguments (cf.~\cite{ahlswede-cai,anderson, harperbook}) imply the following.
\begin{proposition} 
	\proplab{compressed}
	Let $\dim,m$ be positive integers with $m \leq 2^{\dim}$. Among all subsets $\A$ of $\bits$ of size $m$, the maximum of $\er(\A)$ is attained where $\A$ is a left-compressed down-set.
\end{proposition}

\begin{proposition} 
	\proplab{max-weight}
	Let $\A \subseteq \bits$ be a down-set. For every $x\in \A$, we have $|x| \leq \floor{\log |\A|}$.
\end{proposition}
\begin{proof}
	As $x \in \A$, we also have $y \in \A$ for all $y \subseteq x$. The number of such $y$ is $2^{|x|} \leq |\A|$.
\end{proof}

\begin{remark} \propref{compressed} and \propref{max-weight} imply $\mathsf{e}_{\leq 1}(\A) \leq \floor{\log |\A|}\cdot |\A|$. Indeed, for a down-set $\A$, we have $\mathsf{e}_{\leq 1}(\A) = \sum_{x\in \A} |x| \leq |\A| \cdot \floor{\log |\A|}$. This approximates, up to a factor of two, the optimal bound $\mathsf{e}_{\leq 1}(\A) \leq (1/2)\cdot|\A| \cdot \floor{\log |\A|}$ mentioned above~\cite{bernstein, harper, hart, lindsey}.
\end{remark}

We also make use of the following technical result to bound sums of binomial coefficients. The proof  of this proposition can be found in the appendix.

\begin{proposition}
	\proplab{binomial-bound2}
	For all $m \in \mathbb{N}\cup \{0\}, \lambda \in [0,1), K \in \mathbb{R}^+$ we have for $m \neq 0$

\begin{equation*}
\left( \frac{K}{m}\right)^m + \left( \frac{K}{m+1}\right)^{m+1} \geq \left( \frac{K}{m+ \lambda}\right)^{m+ \lambda},
\end{equation*}
and for $m = 0$, we have
$
1 + K \geq \left( \frac{K}{\lambda}\right)^{\lambda}.
$

\end{proposition}

\section{The distance two case}\seclab{distance-two}
\seclab{disttwo}
The special case of our theorem for $r=2$ has a fairly simple proof and a tighter bound. 
\begin{theorem}
	\thmlab{distance-two}
	Let $\A \subset \{0,1\}^{\dim}$ satisfy $1 \leq \log |\A| < \dim$. 
	Then
	$$
	\mathsf{e}_{\leq 2}(\A) 
	\leq \dim 
	\cdot \ell'
	\cdot |\A|,
	$$
	where 
	$\ell': = 
	\min\left\{
	\left\lceil 
	\frac{\log |\A|}{\log \dim - \log \log |\A|} 
	\right\rceil,
	\floor{\log |\A|}
	\right\}$.
\end{theorem}

Using an observation of Ahlswede and Cai~\cite{ahlswede-cai}, we reduce the problem to bounding the ``sum of ranks'' of elements in $\A$. We provide a proof for completeness. Define the rank of $x$ as
$$\|x\| := \sum_{j \in [\dim]} jx_j = \sum_{j \in x}j.$$
\begin{lemma} 
	\lemlab{norm}
	Let $\A$ be a left-compressed down-set. Then,\  $ \mathsf{e}_{\leq 2}(\A) = \displaystyle \sum_{x \in \A} \norm x .$
\end{lemma}
\begin{proof} Notice that $\{x,y\} \in \mathsf{E}_{\leq 2}(\A)$ implies that either $\|y\| < \|x\|$ or vice versa. We fix $x \in \A$ and count $y$ such that
	$\|y\| < \|x\|$. Assume that $x \neq \emptyset, \{1\}$, or the bound is trivial. We separate the cases $|y| = |x|$ and $|y| < |x|$. In the first case, we count~$y$ of the form $y = x \cup \{i\} \setminus \{j\}$, where $i < j$, $j \in x$ and $i \notin x$. The number of such $y$ is exactly
	\[
	\sum_{j \in x} \Big(j - 1 - |\{i \in x:\ i < j\}|\Big) \ = \ \| x \| - \binom{|x|+1}{2}.
	\]
	For the second case, with $|y| < |x|$, there are $\binom{|x|+1}{2}$ choices for $y$ of the form $y = x \setminus \{i,j\}$ or $y = x \setminus \{i\}$, where $i,j \in x$. As we have assumed that $\A$ is a left-compressed down-set, the counted pairs in both cases are in $\mathsf{E}_{\leq 2}(\A)$. Summing over $x\in \A$ completes the proof.
\end{proof}

To obtain \thmref{distance-two} we use the left-compressedness and down-set conditions on $\A$ to find an upper bound of $\|x\|$ for each $x \in \A$ which depends only on $|\A|$ and $n$. The theorem then follows from summing these upper bounds over $x \in \A$. 
\begin{lemma}
	\lemlab{norm-bound} Let $\A\subset \bits$ be a left-compressed down-set with $|\A| \geq 2$. For any $x \in \A$, 
	\begin{equation*} 
	\|x\| \leq \dim \cdot  \ell',
	\end{equation*}
	where $\ell' = 
	\min\left\{
	\left\lceil 
	\frac{\log |\A|}{\log \dim - \log \log |\A|} 
	\right\rceil,
	\floor{\log |\A|}
	\right\}$
\end{lemma}

Assuming this lemma, we now complete the proof of \thmref{distance-two}. 

\begin{proof}[Proof of \thmref{distance-two}] Applying \propref{compressed}, we may assume that $\A$ is a left-compressed down-set. 
	Then, \lemref{norm} and \lemref{norm-bound} together imply the desired bound:
	\[ 
	\mathsf{e}_{\leq 2}(\A) 
	= \sum_{x \in \A} \| x\|  
	\leq \dim \cdot  \ell' \cdot |\A|.
	\]
\end{proof}

\noindent
\textbf{Approximately solving the Kleitman-West problem.} \thmref{distance-two} has the following immediate corollary for the isoperimetric problem on the Kleitman-West graph, i.e., the graph on ${[n] \choose k}$ where two $k$-element sets are joined if they have symmetric difference of size two. For $\A \subset {[n] \choose k}$, let $\mathsf{e}(\A)$ denote the number of edges of this graph induced by $\A$.

\begin{theorem}
	\thmlab{kwc}
	Let $\A \subset {[n] \choose k}$ be nonempty. 
	Then
	$$
	\mathsf{e}(\A) 
	\leq \dim 
	\cdot \ell' 
	\cdot |\A|,
	$$
	where 
	$\ell': = 
	\min\left\{
	\left\lceil 
	\frac{\log |\A|}{\log \dim - \log \log |\A|} 
	\right\rceil,
	\floor{\log |\A|}
	\right\}$.
\end{theorem}

We remark that \thmref{kwc} is sharp up to a factor of $2+o(1)$. This is evidenced by the families defined by
$\left\{x \in {[n] \choose k}:\ [s] \subset x\right\}$ for $k=o(n)$ and $s \in \mathbb{N}$.
\newcommand{\leqB}{\leq_\b}

\subsection{Proof of \lemref{norm-bound}}
	\propref{max-weight} implies that $|x| \leq \floor{\log |\A|}$, and thus, $\|x\| \leq \dim|x| \leq \dim\floor{\log |\A|}$. Therefore, we may assume that we are in the case where $\ell' =\ceil{\frac{\log |\A|}{\log \dim - \log \log |\A|}} < \floor{\log |\A|}$. We note for later use that since $\ell' = \ceil{\frac{\log |\A|}{\log \dim - \log \log |\A|}} < \floor{\log|\A|}$, we have

	\begin{equation} \label{usefulbound}
		2 < \frac{n}{\log |\A|}.
	\end{equation}
	
	We use the fact that $\A$ is a left-compressed down-set to lower bound the number of $y \in \A$ that are guaranteed in $\A$ by the existence of $x \in \A$. To this end, define $\b' := \left\lfloor \frac{\dim\ell'}{\log |\A|}\right\rfloor$, and let $x = x' \cup x''$, where $x' \subseteq \{1,\ldots, \b'\}$ and $x'' \subseteq \{\b' + 1,\ldots, \dim\}$ correspond to the integers in $x$ with values at most $\b'$ and at least $\b' +1$, respectively (so that $|x| = |x'| + |x''|$). We will show that 
	\begin{equation*}\label{norm-bound-in-lemma}
	\|x\| \leq \b'|x'| + \dim|x''| \leq \dim\ell'.
	\end{equation*}
	Notice that if $|x''| = 0$, then 
	$ 
	\|x\| = \b'|x'| = \b'|x| \leq \b' \log |\A| \leq \dim\ell',
	$ where the inequalities use \propref{max-weight} and the definition of $\b'$. Thus, we may assume that $|x'| \leq |x|-1$ and $|x''| \geq 1$. 
	
	Consider $y \in \bits$ of the form $y = y' \cup y''$, where $y' \subseteq x'$, $y'' \subseteq \left([\beta']\setminus x'\right) \cup x''$, and $|y''| \leq |x''|$. We claim every $y$ of this form is in $\A$. Indeed, this follows directly from the left-compressed down-set assumption. To count such $y \in \A$, first define $\eps_x \in [0,1)$ as the real number satisfying $2^{|x'|} = |\A|^{\eps_x}$. We will show $|x''| \leq (1-\eps_x)\ell'$. 
	Clearly, there are $2^{|x'|} = |\A|^{\eps_x}$ choices for $y' \subseteq x'$ and 
	\[
	\text{\# of choices for } y'' = \sum_{j = 0}^{|x''|} \binom{\b' + |x''| - |x'|}{j},
	\]
	where the $\jth$ term counts $y''$ with $|y''| = j$.  Since the choice of $y'$ is independent of $y''$, we know that the sum above must be at most $|\A|^{1-\eps_x}$, otherwise we would have guaranteed more than $|\A|$ distinct $y$ in $\A$.

	Aiming for a contradiction, we suppose that $|x''| \geq \ceil{(1-\eps_x)\ell'}$ and $\eps_x \leq 1 - 1/\ell'$. It is a standard fact that for $a,b \in \mathbb{N}$ where $a \geq b \geq 1$ we have $\binom{a}{b} \geq \left( \frac{a}{b}\right)^b$. This fact and the assumption $|x''| \geq \ceil{(1-\eps_x)\ell'}$ imply the lower bound

\begin{align}
	\sum_{j = 0}^{|x''|} \binom{\b' + |x''| - |x'|}{j} &\geq 
		\left( \frac{\beta' + |x''| - |x'|}{\left \lceil (1- \eps_x)\ell' \right \rceil }\right)^{\left \lceil (1- \eps_x)\ell' \right \rceil} + \left( \frac{\beta' + |x''| - |x'|}{\left \lceil (1- \eps_x)\ell' \right \rceil-1 }\right)^{\left \lceil (1- \eps_x)\ell' \right \rceil - 1}\\
	\label{eqn:sumlowbound}
	&\geq \left ( \frac{\b' + |x''| - |x'|}{(1- \eps_x)\ell'}\right)^{(1-\eps_x)\ell'},
\end{align}
where the final inequality follows by applying \propref{binomial-bound2}.

	 We note that if $a>2$ then $\frac{\left \lfloor a\right \rfloor}{\log a} \geq \frac{2}{\log(3)}$. Using our observation in equation (\ref{usefulbound}) we apply this fact to the definition of $\b'$ to see
	\begin{equation}\label{eqn:betafact}
	\beta' = \left \lfloor \frac{n\ell'}{\log |\A|} \right \rfloor\geq \left \lfloor \frac{n}{\log|\A|}\right \rfloor \left \lceil \frac{\log |\A|}{\log n - \log \log |\A|} \right \rceil  \geq \frac{2}{\log(3)} \log |\A|.
	\end{equation}
\noindent
Observe that (\ref{eqn:betafact}) and the fact $|x'| = \eps_x \log |\A|$ together imply $\b' - |x'| \geq (1-\frac{\log 3}{2}\eps_x)\b'$. 

We now split into the following cases:

(1) $|x'| \geq 4$,

(2) $2 \leq |x'| \leq 3$,

(3) $|x'| \leq 1$.

\paragraph{Case (1): $|x'| \geq 4$.} We note that $|x'| \geq 4$ is equivalent to $\eps_x \log |\A| \geq 4$ and this implies $\eps_x > \frac{\log 3}{(2-\log 3) \log |\A|}$, which after rearranging is equivalent to $\frac{2- \log 3}{2}\eps_x > \frac{\log 3}{2 \log |\A|}$. Using inequality (\ref{eqn:betafact}), and that $1/(1-\eps_x) \geq 1$, we see $\frac{2- \log 3}{2(1-\eps_x)}\eps_x > \frac{1}{\b'}$. Now, by the definition of~$\b'$, the right hand side of this inequality trivially satisfies

\begin{equation}\label{betatrick}
\frac{1}{\b'} \geq \frac{\frac{n \ell'}{ \log |\A|} - \b'}{\b'},
\end{equation}
so rearranging we see that
\begin{equation*}
\left( \frac{(1 - \frac{\log 3}{2}\eps_x) \b'}{(1-\eps_x) \ell'} \right) = \left( 1 + \frac{2- \log 3}{2(1- \eps_x)}\eps_x\right) \frac{\beta'}{\ell'} > \frac{n}{\log |\A|}.
\end{equation*}
Using our observation that $\b' - |x'| \geq (1-\frac{\log 3}{2}\eps_x)\b'$ we arrive at
\begin{equation*}
\frac{\b' +|x''| - |x'|}{(1-\eps_x)\ell'} > \frac{n}{\log |\A|}.
\end{equation*}
Substituting this into the lower bound (\ref{eqn:sumlowbound}) we see
\begin{equation*}
\sum_{j = 0}^{|x''|} \binom{\b' + |x''| - |x'|}{j} > \left( \frac{n}{\log |\A|} \right)^{(1-\eps_x)\ell'} \geq |\A|^{1- \eps_x},
\end{equation*}
giving the required contradiction. 

\paragraph{Case (2): $2 \leq |x'| \leq 3$.} As $|x'| \leq 3$ we have $|x''| \geq 1 \geq |x'|/3$, and so
$$\b' + |x''| - |x'| \geq \b' - 2|x'|/3.$$
We combine this with fact (\ref{eqn:betafact}) to get $\b' + |x''| - |x'| \geq (1- \frac{\log3}{3}\eps_x) \b'$. Therefore
\begin{equation} \label{strongerbound}
\frac{\b' + |x''| - |x'|}{(1-\eps_x)\ell'} \geq \left(\frac{1 - \frac{\log 3}{3} \eps_x}{1- \eps_x}\right)\frac{\b'}{\ell'} = \left( 1 + \frac{3 - \log 3}{3(1- \eps_x)}\eps_x\right) \frac{\b'}{\ell'}.
\end{equation}
Now, since $|x'| \geq 2$ is equivalent to $\eps_x \log |\A| \geq 2$ we see $\eps_x > \frac{3 \log 3}{2(3 - \log 3) \log |\A|}$ which after rearranging is equivalent to $\frac{3- \log3}{3} \eps_x > \frac{\log 3}{2 \log |\A|}$. Using inequality (\ref{eqn:betafact}), and that $1/(1-\eps_x) \geq 1$, we see $\frac{3-\log 3}{3(1-\eps_x)}\eps_x > \frac{1}{\b'}$. Now, as in the previous case, we appeal to equation (\ref{betatrick}) and rearrange to find
\begin{equation*}
\left( 1 + \frac{3-\log 3}{3(1-\eps_x)}\eps_x \right) \frac{\beta'}{ \ell'} > \frac{n}{\log |\A|}.
\end{equation*}
Combining this with the inequality (\ref{strongerbound}) we find again $\frac{\b' + |x''| - |x'|}{(1- \eps_x)\ell'} > \frac{n}{\log |\A|}$. Substituting this into the lower bound (\ref{eqn:sumlowbound}) gives the required contradiction.

\paragraph{Case (3): $|x'| \leq 1$.} Suppose first that $|x'| = 0$, and so $\eps_x = 0$. Then by assumption $|x''| \geq \left \lceil \ell' \right \rceil$. Hence

\begin{align*}
\sum_{j = 0}^{|x''|} \binom{\b' + |x''| - |x'|}{j} &\geq \binom{\b' + |x''|}{\ell'} + \binom{\b' +|x''|}{\ell' - 1} = \binom{\b' + |x''| + 1}{\ell'}\\
& \geq \left( \frac{\b' + |x''| + 1}{\ell'}\right)^{\ell'},
\end{align*}
and since $\b' +|x''|+1 = \left \lfloor \frac{n\ell'}{\log |\A|} \right \rfloor+ |x''| + 1 > \frac{n\ell'}{\log |\A|}$ we see that

\begin{equation*}
\sum_{j = 0}^{|x''|} \binom{\b' + |x''| - |x'|}{j} > \left( \frac{n}{\log |\A|}\right)^{\ell'} \geq |\A|,
\end{equation*}
providing the required contradiction.

Secondly, we suppose that $|x'| = 1 \leq |x''|$. In this case, we have
\begin{equation*}
\sum_{j = 0}^{|x''|} \binom{\b' + |x''| - |x'|}{j} \geq \left( \frac{\b' + |x''| - |x'|}{(1-\eps_x)\ell'} \right)^{(1-\eps_x)\ell'} \geq \left( \frac{\beta'}{(1-\eps_x)\ell'} \right)^{(1-\eps_x)\ell'}.
\end{equation*}
Now $|x'| \geq 1$ is equivalent to $\eps_x \log |\A| \geq 1$ which implies $\eps_x > \frac{\log 3}{2 \log |\A|}$. Using inequality (\ref{eqn:betafact}) we see $\eps_x > 1/\b'$, which implies $\frac{\b'}{1- \eps_x} > \frac{n\ell'}{\log |\A|}$. Thus, if $|x'| = 1 \leq |x''|$ then
\begin{equation*}
\sum_{j = 0}^{|x''|} \binom{\b' + |x''| - |x'|}{j} > \left( \frac{n}{\log |\A|} \right)^{(1-\eps_x)\ell'} = |\A|^{(1-\eps_x)},
\end{equation*}
again giving a contradiction.

Since in every case we arrive at a contradiction, the assumption $|x''| \geq \ceil{(1-\eps_x)\ell'}$ is false and so we must have $|x''| \leq \ceil{(1-\eps_x)\ell'} - 1 < (1-\eps_x)\ell'$, and thus we conclude that
	\[
	\|x\| \ \leq\ \b'|x'| + \dim|x''| 
	\ =\ \b'  \eps_x \log |\A| + \dim|x''|
	\ \leq\ \eps_x \dim\ell' +  (1-\eps_x)\dim\ell' \ =\ \dim\ell'.
	\]

\section{The general case for even distances}

In this section, we prove \thmref{even-power-approx}, which, using the notation defined in \secref{notation}, is equivalent to the statement that if $\A \subset \{0,1\}^{\dim}$ and $t \in \mathbb{N}$ with $t \leq \log |\A|$, then
	$$
	|\mathsf{E}_{\leq 2t}(\A)| := \mathsf{e}_{\leq 2t}(\A) \ \leq\  
	\left(\frac{8e}{t}\right)^{2t} 
	\cdot \left(\dim \cdot \ell  \right)^t 
	\cdot |\A|,
	$$
	where 
	$$\ell = \ell(\A) := 
\min\left\{
\left\lceil 
\frac{2\log |\A|}{\log n - \log \log |\A|} 
\right\rceil,
\floor{\log |\A|}
\right\}.$$
We start with some more notation. For $(b,a) \in \mathbb{Z}_{\geq 0}^2$, let 
$$\mathsf{E}_{(b,a)}(\A) := \{\{x,y\} \in \mathsf{E}_{\leq 2t}(\A):\ |x \setminus y| = b,\ |y \setminus x|=a\}.$$
and define
$\mathsf{e}_{(b,a)}(\A) := |\mathsf{E}_{(b,a)}(\A)|.$
Letting 
\[
\mathcal{\calU} = \{(b,a) \in \mathbb{Z}_{\geq 0}^2:\ b \geq a\ \text{and}\ b+a \leq 2t\},
\] 
observe that we can decompose $\mathsf{E}_{\leq 2t}(\A)$ as a disjoint union 
\[ 
\mathsf{E}_{\leq 2t}(\A) = \bigcup_{(b,a) \in \calU} \mathsf{E}_{(b,a)}(\A),
\] 
and in particular, this implies,
\begin{equation} \label{eqn:sum-ab-bound}
\mathsf{e}_{\leq 2t}(\A) 
= \sum_{(b,a) \in \calU} \mathsf{e}_{(b,a)}(\A).
\end{equation}

Our strategy will be to prove upper bounds on $\mathsf{e}_{(b,a)}(\A)$, and then combine these to obtain the theorem. We will need a variant of the bound on $|x''|$ from the proof of \lemref{norm-bound}. In what follows, we express our results using integers $\ell := \ell(\A)$ and $\b := \b(\A)$, defined in the next proposition. 
We also define $\ell_x := |x \cap \{\b+1,\ldots,\dim\}|$ for $x \in \A$. Intuitively, $\b$ is the threshold for `big' elements; $\ell_x$ is the number of these `big' elements; and, we will show that $\ell_x \leq \ell$.

\begin{proposition}\proplab{sB-bounds} Let $n \geq 2$ and $\A \subset \bits$ be a down-set with $|\A| \geq 2$. Let 
	$$
	\ell = 
	\min\left\{
	\left\lceil 
	\frac{2\log |\A|}{\log \dim - \log \log |\A|} 
	\right\rceil,
	\floor{\log |\A|}
	\right\},
	\quad 
	\b = 
	\left\lfloor \left(\frac{ \dim}{\log |\A|}\right)^{1/2} \ell \right \rfloor.
	$$ 
	For any $x \in \A$, we have the following: 

	(i) $|x|\cdot \b \leq \dim\ell$,

	(ii) $\b^2 \leq \dim\ell$,

	(iii) $\log^{2} |\A| \leq \frac{n}{n-1} \dim \ell$,

	(iv) $|x|^2 \leq n\ell$,

	(v) $\left \lfloor \log |\A| \right \rfloor \log |\A| \leq n\ell$.
\end{proposition}
\begin{proof} Parts (i) and (ii) follow immediately from \propref{max-weight}, the fact that $\log |\A| \leq n$ and the definitions of $\b$ and $\ell$.

For part (iii), since $ \log (n/\log |\A|) \leq n/\log |\A|$ we see that
\begin{equation*}
\log^2 |\A| \leq \frac{n \log |\A|}{\log(n/\log |\A|)}.
\end{equation*}
Hence, if $\ell = \left \lceil \frac{2 \log |\A|}{\log n - \log \log |\A|} \right \rceil$ then $\ell \geq \frac{\log |\A|}{\log(n/\log |\A|)}$ and we see the stronger statement $\log^2 |\A| \leq n\ell$ holds, and we note this for later. On the other hand, if $\ell = \left \lfloor \log |\A| \right \rfloor < \left \lceil \frac{2 \log |\A|}{\log n - \log \log |\A|} \right \rceil$, then $n\ell \geq n( \log |\A| - 1)$, so it is sufficient to show $\frac{n}{n-1} n (\log |\A| - 1) \geq \log^2 |\A|$, which is true if and only if $\frac{n}{n-1} \leq \log |\A| \leq n$.

Therefore, the only remaining cases to check are when $1 \leq \log |\A| < \frac{n}{n-1}$. Under this assumption, $\ell = 1$ and $\log^2 |\A| < \left(\frac{n}{n-1}\right)^2$, so as $n \geq 2$ we see that $\frac{n^2}{n-1} \geq \left(\frac{n}{n-1}\right)^2$ which in turn shows $\frac{n}{n-1} n\ell \geq \log^2 |\A|$ as required.

For part (iv) let $x\in \A$. We have already seen $|x| \leq \left \lfloor \log |\A|\right \rfloor$ and $|x| \leq n$ is trivial. If $\ell = \left \lceil \frac{2\log |\A|}{\log n - \log \log |\A|} \right \rceil$, we recall that $\log^2|\A| \leq n\ell$, and so $|x|^2 \leq n\ell$. On the other hand, if $\ell=\left \lfloor \log |\A|\right \rfloor$, then $|x|^2 \leq n\ell$. This proves (iv).

Finally, for part (v), again recall that if $\ell = \left \lceil \frac{2 \log |\A|}{\log n - \log \log |\A|} \right \rceil$ then $\left \lfloor \log |\A| \right \rfloor \log |\A| \leq \log^2 |\A| \leq n\ell$ and so $\left \lfloor \log |\A| \right \rfloor \log |\A| \leq n \ell$ follows. On the other hand if $\ell = \left \lfloor \log |\A| \right \rfloor$, then as $\log |\A| \leq n$ we see $\left \lfloor \log |\A| \right \rfloor \log |\A| \leq n\ell$, completing the proof of (v).
\end{proof}

\begin{lemma}
	\lemlab{few-large} \lemlab{sum-bound} 
	Let $\A \subset \bits$, $|\A| \geq 2$ be a left-compressed down-set. 
	If $x \in \A$, then 
	$\ell_x \leq \ell.$
\end{lemma}
\begin{proof}
	\propref{max-weight} implies $|x| \leq \floor{\log |\A|}$, and clearly $\ell_x \leq |x|$, so we may assume that we are in the case when $\ell =\ceil{\frac{2\log |\A|}{\log \dim - \log \log |\A|}}$.
	Let $x = x' \cup x''$ where $x' \subseteq  \{1,\ldots, \b\}$ and $x'' \subseteq  \{\b + 1,\ldots, \dim\}$. By definition, $|x''| = \ell_x$, and since $\A$ is a down-set, we know that $x'' \in \A$. Suppose $y \subseteq [\beta] \cup x''$ with $|y| \leq \ell_x$. As $\A$ is left-compressed and a down-set $y \in \A$. Counting such $y$ we have
	\begin{equation}
	\label{ycount}
	|\A| \geq \sum_{j=0}^{\ell_x} \binom{\b +\ell_x}{j}.
	\end{equation} 
	Suppose now, for a contradiction, that $\ell_x \geq \ell+1$. Then clearly

	\begin{equation*}
		\sum_{j=0}^{\ell_x} \binom{\b +\ell_x}{j} \geq \binom{\b + \ell_x}{\ell} + \binom{\beta + \ell_x}{\ell-1}.
	\end{equation*}
Applying \propref{binomial-bound2} to this inequality and combining with the lower bound (\ref{ycount}) we find that
	\begin{equation} \label{Alowbound}
		|\A| \geq  \left( \frac{\beta + \ell_x}{ 2\log |\A| / \log(n/\log|\A|)}\right)^{2\log |\A| / \log(n/\log |\A|)}.
	\end{equation}
Now, since $\ell_x \geq \ell+ 1$ it is clear that
	\begin{equation*}
	\frac{\beta + \ell_x}{ 2\log |\A| / \log(n/\log|\A|)} \geq \frac{\beta + 1+ \ell}{ 2\log |\A|}  \cdot \log \left(\frac{n}{\log|\A|}\right),
	\end{equation*}
	and so by substituting the definition of $\b$ into this inequality, we see that
	\begin{equation*}
	\frac{\beta + \ell_x}{ 2\log |\A| / \log(n/\log|\A|)} \geq  \frac{ \left( \left( \frac{n}{\log |\A|} \right)^{1/2} + 1 \right)  \cdot\ell}{ 2\log |\A|}  \cdot \log \left(\frac{n}{\log|\A|}\right) > \left(\frac{n}{\log |\A|} \right)^{1/2}.
	\end{equation*}
	From this, and equation (\ref{Alowbound}) we see that
	\begin{equation*}
		|\A| > \left(\frac{n}{\log|\A|}\right)^{\log|\A|/\log(n/\log|\A|)} = |\A|,
	\end{equation*}
	which is a contradiction. We therefore deduce that $\ell_x \leq \ell$.
\end{proof}

In what follows, let $\A \subseteq \bits$ be a left-compressed down-set with $1 \leq \log |\A| < \dim$. Let $\ell,\b$ be defined as in \propref{sB-bounds}. Recall that $\ell_x = |x \cap \{\b+1, \ldots, \dim\}|$ equals the number of large elements in $x\in \A$. In our proofs, it will be helpful to order  $\bits$ based on $\ell_x$. In particular, we upper bound $\mathsf{e}_{(b,a)}(\A)$ by partitioning the pairs $\{x,y\} \in \mathsf{E}_{(b,a)}(\A)$ into two sets, based on the cases $\ell_y \leq \ell_x$ and $\ell_y > \ell_x$. By the definition of $\mathsf{E}_{(b,a)}(\A)$, with $b \geq a$, we always have $|x| \geq |y|$. Ordering based on $\ell_x$ and $\ell_y$ enables us to use different arguments in the two cases: when $\ell_y \leq \ell_x$, we count pairs based on $x$, and when $\ell_y > \ell_x$, we count pairs based on $y$. 

\subsection{The case $\ell_y \leq \ell_x$}

\begin{lemma} \lemlab{eba-less}
	Let $b,a$ be nonnegative integers with $b \geq a$ and $1 \leq b+a \leq 2\log |\A|$. 
	\begin{itemize}
		\item If $b+a$ is even, then \[
		|\{\{x,y\} \in \mathsf{E}_{(b,a)}(\A):\ \ell_y \leq \ell_x\}|
		\leq 
		\left(\frac{4\sqrt{2}e}{b+a}\right)^{(b+a)} \cdot (\dim\cdot\ell)^{(b+a)/2}
		\cdot |\A|.
		\]
		\item If $b+a$ is odd, then 
		\[
		|\{\{x,y\} \in \mathsf{E}_{(b,a)}(\A):\ \ell_y \leq \ell_x\}|
		\leq 
		\left(\frac{4\sqrt{2}e}{b+a}\right)^{b+a} \cdot (\dim\cdot \ell)^{(b+a-1)/2} \cdot \log |\A| \cdot |\A|.
		\]
	\end{itemize}
\end{lemma}
\begin{proof}
	Fix $x \in \A$. For each $p \in [a] \cup \{0\}$, we will bound the number of $y \in \bits$ such that $\{x,y\} \in \mathsf{E}_{(b,a)}(\A)$ and $\ell_y \leq \ell_x$ and 
	$
	|(y \setminus x) \cap \{\b+1,\ldots,\dim\}|=p. 
	$ 
	We claim that the number of such $y$ is at most 
	\begin{equation}\label{eqn:y-choices}
	\binom{\dim-\b-\ell_x}{p} \binom{\ell_x}{p}
	\binom{\b - |x| + \ell_x}{a-p} \binom{|x|}{b-p}.
	\end{equation}
	Indeed, the first two factors count the ways to replace $p$ elements in $x$ with $p$ new elements that are larger than $\b$, and the final two factors count the ways to replace $b-p$ elements in~$x$ with $a-p$ new elements that are at most $\b$.  
	
	Recall that \lemref{few-large} implies that $\ell_x \leq \ell$. Therefore, the quantity in (\ref{eqn:y-choices}) is at most 
	\begin{equation}\label{eqn:bap-less}
	\binom{\dim}{p} \binom{\ell}{p}
	\binom{\b}{a-p} \binom{|x|}{ b-p}
	\ \leq\ 	\frac{(\dim \ell)^{p} \cdot \b^{a-p}|x|^{b-p}}{(p!)^2\cdot (a-p)! \cdot (b-p)!} .
	\end{equation}
	
	We note that for $i,j \geq 0$ we have $i^i j^j \geq \left( \frac{i+j}{2}\right)^{i+j}$. Indeed, taking logs and dividing by~2, this is equivalent to
$$\tfrac{1}{2}(i \log i + j \log j) \geq \tfrac{i+j}{2}\log\left( \tfrac{i+j}{2}\right),$$
which follows from the convexity of the function $z \mapsto z \log z$. Hence, we may bound from below the denominator of the right-hand side of equation (\ref{eqn:bap-less}) as follows:
	\begin{align} 
		(p!)^2 \cdot (a-p)! \cdot (b-p)! &\geq \frac{p^{2p} \cdot (a-p)^{a-p} \cdot (b-p)^{b-p}}{e^{b+a}} \quad \text{(by Stirling's approximation)}\\\label{eqn:denominatorlowerbound}
		& \geq \left(\frac{b+a}{4e}\right)^{b+a} \quad \text{(by two applications of } i^i j^j \geq \left( \frac{i+j}{2}\right)^{i+j}\text{)}.
	\end{align}

	We now break the bounding of (\ref{eqn:bap-less}) into two cases, based on the parity of $b+a$. For both cases, recall that \propref{sB-bounds} implies that $\b|x| \leq \dim\ell$ and $\b^2 \leq \dim\ell$ and $|x|^2 \leq \dim\ell$.
	
	\paragraph{The case where $b+a$ is even.} We bound the numerator of the RHS of  (\ref{eqn:bap-less}) by 
	\[
	(\dim \ell)^{p} \cdot \b^{a-p}|x|^{b-p}
	\leq 
	(\dim \ell)^{p} \cdot (\dim\ell)^{(a-p)/2}\cdot(\dim\ell)^{(b-p)/2}
	= 
	(\dim \ell)^{(b+a)/2}.
	\]
	Summing the above bound on  (\ref{eqn:bap-less}) over $p \in [a] \cup \{0\}$ and employing (\ref{eqn:denominatorlowerbound}), we obtain 
	\begin{eqnarray*} 
		|\{y \in \A : \{x,y\} \in \mathsf{E}_{(b,a)}(\A),\ \ell_y \leq \ell_x\}|
		&\leq&
		\sum_{p=0}^a  \frac{\left(\dim\ell\right)^{(b+a)/2}}{(p!)^2 \cdot (b-p)! \cdot (a-p)!}\\
		&\leq& 
		(a+1) \cdot \frac{\left(\dim\ell\right)^{(b+a)/2}(4e)^{(b+a)}}{(b+a)^{b+a}} \\
		&\leq& 
		\frac{\left(\dim\ell\right)^{(b+a)/2}(4\sqrt{2}e)^{(b+a)}}{(b+a)^{b+a}},
	\end{eqnarray*}
	where the last inequality uses the fact that $(a+1) \leq (\sqrt{2})^{b+a}$, leading to the factor $(4\sqrt{2}e)^{(b+a)}$.
	\paragraph{The case where $b+a$ is odd.} In this case, we have $b \geq a+1 \geq p+1$. We recall that $|x| \leq \log |\A|$,  and we upper bound the numerator of the RHS of  (\ref{eqn:bap-less}) by 
	\[
	(\dim \ell)^{p} \cdot \b^{a-p}|x|^{b-p}
	\leq 
	(\dim \ell)^{p} \cdot (\dim\ell)^{(a-p)/2}\cdot(\dim\ell)^{(b-p-1)/2} \cdot \log |\A|
	= 
	(\dim \ell)^{(b+a-1)/2}\cdot \log |\A|.
	\]
	Summing the above bound on (\ref{eqn:bap-less}) over $p \in [a] \cup \{0\}$ and employing (\ref{eqn:denominatorlowerbound}),  we obtain
	\begin{eqnarray*} 
		|\{y \in \A : \{x,y\} \in \mathsf{E}_{(b,a)}(\A),\ \ell_y \leq \ell_x\}|
		&\leq&
		\sum_{p=0}^a  \frac{\left(\dim\ell\right)^{(b+a-1)/2}\cdot \log |\A|}{(p!)^2 \cdot (b-p)! \cdot (a-p)!}\\
		&\leq& 
		\frac{\left(\dim\ell\right)^{(b+a-1)/2}(4\sqrt{2}e)^{(b+a)}\cdot \log |\A|}{(b+a)^{b+a}}.
	\end{eqnarray*} 
	In both even and odd cases, summing over $x \in \A$
	completes the proof. 
\end{proof}

\subsection{The case $\ell_y > \ell_x$}

\begin{lemma} \lemlab{eba-more}
	Let $b,a$ be nonnegative integers with $b \geq a$ and $1 \leq b+a \leq 2\log |\A|$. 
	\begin{itemize}
		\item If $b+a$ is even, then \[
		|\{\{x,y\} \in \mathsf{E}_{(b,a)}(\A):\ \ell_y > \ell_x\}|
		\leq 
		\left(\frac{4\sqrt{2}e}{b+a}\right)^{(b+a)} \cdot (\dim\cdot\ell)^{(b+a-2)/2}\cdot \ell \b
		\cdot |\A|.
		\]
		\item If $b+a$ is odd, then 
		\[
		|\{\{x,y\} \in \mathsf{E}_{(b,a)}(\A):\ \ell_y > \ell_x\}|
		\leq 
		\left(\frac{4\sqrt{2}e}{b+a}\right)^{b+a} \cdot (\dim\cdot \ell)^{(b+a-1)/2} \cdot \ell \cdot |\A|.
		\]
	\end{itemize}
\end{lemma}
\begin{proof}
	Fix $y \in \A$. For each $p \in [a]$, we will bound the number of $x \in \bits$ such that $\{x,y\} \in \mathsf{E}_{(b,a)}(\A)$ and $\ell_y > \ell_x$ and 
	$
	|(x \setminus y) \cap \{\b+1,\ldots,\dim\}|=p-1. 
	$ 
	We claim that the number of such $x$ is at most 
	\begin{equation}\label{eqn:x-choices}
	\binom{\dim-\b-\ell_y}{p-1} \binom{\ell_y}{p}
	\binom{\b - |x| + \ell_y}{b-p+1} \binom{|y|}{a-p}.
	\end{equation}
	Indeed, the first two factors count the ways to replace $p$ elements in $y$ with $p-1$ new elements that are larger than $\b$, and the final two factors count the ways to replace $a-p$ elements in~$y$ with $b-p+1$ new elements that are at most $\b$.  
	
	Recall that Lemma \ref{lemma:few-large} implies that $\ell_y \leq \ell$. Thus, the quantity in (\ref{eqn:x-choices}) is at most 
	\begin{equation}\label{eqn:bap-more}
	\binom{\dim}{p-1} \binom{\ell}{p}
	\binom{\b}{b-p+1} \binom{|y|}{a-p}
	\ \leq\ 	
	\frac{(\dim \ell)^{p-1} \cdot \ell \cdot \b^{b-p+1} \cdot |y|^{a-p}}
	{(p-1)!\cdot p!\cdot (b-p+1)! \cdot (a-p)!} .
	\end{equation}
Similarly to in the proof of \lemref{eba-less} (i.e., by applying Stirling's approximation and the fact $i^i j^j \geq (\frac{i+j}{2})^{i+j}$), we lower bound the denominator of the right hand side of (\ref{eqn:bap-more}) as follows.

	\begin{align}
		(p-1)! \cdot p! \cdot (b-p+1)! \cdot (a-p)! &\geq \frac{(p-1)^{p-1} \cdot p^p \cdot (a-p)^{a-p} \cdot (b-p+1)^{b-p+1}}{e^{b+a}}\\
		\label{eqn:denominatorbound2}
		& \geq \left( \frac{b+a}{4e} \right)^{b+a}.
	\end{align}
	Recall that \propref{sB-bounds} implies that $\b^2 \leq \dim\ell$ and $|y|^2 \leq \dim\ell$. We now break into two cases, based on the parity of $b+a$. 
	
	\paragraph{The case where $b+a$ is even.} Notice that $\ell_y > \ell_x$ and $|x| \geq |y|$ implies $a \geq 1$ and $b+a \geq 2$. We upper bound the numerator of the RHS of  (\ref{eqn:bap-more}) by 
	\[
	(\dim \ell)^{p-1} \cdot \ell \cdot \b^{b-p+1} \cdot |y|^{a-p}
	\leq 
	(\dim \ell)^{p-1}\cdot \ell \cdot \b \cdot
	(\dim\ell)^{(b-p)/2}\cdot(\dim\ell)^{(a-p)/2}
	= 
	(\dim \ell)^{(b+a-2)/2} \cdot \ell \b.
	\]
Summing our bound on (\ref{eqn:bap-more}) over $p \in [a]$, employing (\ref{eqn:denominatorbound2}), and using that $a \leq (\sqrt{2})^{b+a}$,
	\begin{eqnarray*} 
		|\{x \in \A : \{x,y\} \in \mathsf{E}_{(b,a)}(\A),\ \ell_y > \ell_x\}|
		&\leq&
		\sum_{p=1}^a  \frac{\left(\dim\ell\right)^{(b+a-2)/2} \cdot \ell \b}{p!\cdot (p-1)! \cdot (b-p+1)! \cdot (a-p)!}\\
		&\leq&
		\frac{\left(\dim\ell\right)^{(b+a-2)/2}(4\sqrt{2}e)^{(b+a)}\cdot \b\ell}
		{(b+a)^{b+a}}.
	\end{eqnarray*}
	
\paragraph{The case where $b+a$ is odd.} Notice that $\ell_y > \ell_x$ and $|x| \geq |y|$ implies $a \geq 1$, and in this case, $b \geq a + 1 \geq p+1$. We upper bound the RHS of  (\ref{eqn:bap-more}) by 
	\[
	(\dim \ell)^{p-1} \cdot \ell \cdot \b^{b-p+1} \cdot |y|^{a-p}
	\leq 
	(\dim \ell)^{p-1}\cdot \ell \cdot
	(\dim\ell)^{(b-p+1)/2}\cdot(\dim\ell)^{(a-p)/2}
	= 
	(\dim \ell)^{(b+a-1)/2} \cdot \ell.
	\]
	Summing our bound on (\ref{eqn:bap-more}) over $p \in [a]$, employing (\ref{eqn:denominatorbound2}), and using that $a \leq (\sqrt{2})^{b+a}$,
	\begin{eqnarray*} 
		|\{x \in \A : \{x,y\} \in \mathsf{E}_{(b,a)}(\A),\ \ell_y > \ell_x\}|
		&\leq&
		\sum_{p=1}^a  \frac{\left(\dim\ell\right)^{(b+a-1)/2} \cdot \ell }{p!\cdot (p-1)! \cdot (b-p+1)! \cdot (a-p)!}\\
		&\leq&
		\frac{\left(\dim\ell\right)^{(b+a-1)/2}(4\sqrt{2}e)^{(b+a)}\cdot \ell}
		{(b+a)^{b+a}}.
	\end{eqnarray*}
	In both even and odd cases, summing over $y \in \A$ completes the proof. 
\end{proof}

\subsection{Finishing the proof}

\begin{proof}[Proof of \thmref{even-power-approx}] 
	Recall that $
	\calU := \{(b,a) \in \mathbb{Z}_{\geq 0}^2:\ b \geq a\ \text{and}\ b+a \leq 2t\}.$ Invoking (\ref{eqn:sum-ab-bound}) and using \lemref{eba-less} and \lemref{eba-more}, we will upper bound each term in 
	$$
	\mathsf{e}_{\leq 2t}(\A) 
	= \sum_{(b,a) \in \calU} \mathsf{e}_{(b,a)}(\A).
	$$
	For all $(b,a) \in \calU$, we claim that
	\begin{equation}\label{ab-2t}	
	\frac{\mathsf{e}_{(b,a)}(\A)}{|\A|}  
	\leq \left(\frac{4e}{t}\right)^{2t} (\dim \ell)^t.
	\end{equation}
	Assuming that (\ref{ab-2t}) holds, and using that $|\calU| \leq 2^{2t}$, we have
	\[
	\sum_{(b,a) \in \calU} \frac{\mathsf{e}_{(b,a)}(\A)}{|\A|} 
	\leq |\calU| \cdot \left(\frac{4e}{t}\right)^{2t} (\dim \ell)^t
	\leq \left(\frac{8e}{t}\right)^{2t} (\dim \ell)^t,
	\]
	which implies the bound in the theorem statement.
To prove (\ref{ab-2t}), we will use \propref{sB-bounds} and the fact that $t \leq \left \lfloor \log |\A| \right \rfloor$. When $b+a$ is even, then combining \lemref{eba-less} and \lemref{eba-more} (using $\b\ell \leq \dim\ell$), we have 

	\begin{align*}
e_{(b,a)}(\mathcal{A}) &\leq \left ( \frac{4\sqrt{2} e}{b+a}\right )^{(b+a)}\cdot (n\ell)^{(b+a)/2}\cdot |\mathcal{A}| +\left ( \frac{4\sqrt{2} e}{b+a}\right )^{(b+a)}\cdot (n\ell)^{(b+a-2)/2}\cdot \ell\beta \cdot |\mathcal{A}|\\
& = \left ( \frac{4\sqrt{2} e}{b+a}\right )^{(b+a)} \cdot |\mathcal{A}| \cdot (n\ell)^{(b+a - 2)/2} \cdot \left( n\ell + \ell\beta \right)\\
& \leq 2 \cdot \left ( \frac{4\sqrt{2} e}{b+a}\right )^{(b+a)} \cdot |\mathcal{A}| \cdot (n\ell)^{(b+a)/2} \quad \text{(as } \ell\beta \leq n\ell \text{)}\\
&\leq \left ( \frac{8 e}{b+a}\right )^{(b+a)} \cdot |\mathcal{A}| \cdot (n\ell)^{(b+a)/2} \quad \text{(as } 2 \leq \sqrt{2}^{(b+a)} \text{)}.
\end{align*}
	To verify (\ref{ab-2t}), it suffices to show that the RHS of the above inequality increases with $b+a$ (i.e. that it is maximized over $\calU$ at $b+a = 2t$). Indeed, let $k = b+a \geq 2$. Then, it suffices to show that
	\begin{equation}
	\label{ba-increases}
	\left(\frac{8e}{k-1}\right)^{k-1} \cdot (\dim\cdot\ell)^{k/2 - 1/2}
	\leq
	\left(\frac{8e}{k}\right)^{k} \cdot (\dim\cdot\ell)^{k/2}.
	\end{equation}
After rearranging, we have 
	\begin{equation*}
	\frac{k}{8e} \left(\frac{k}{k-1}\right)^{k-1} 
	\ \leq\  \frac{k}{8} 
	\ \leq\
	(\dim \ell)^{1/2},
	\end{equation*}
	where the first inequality uses that $(\frac{k}{k-1})^{k-1} \leq e$, and the second inequality uses that $(k/8)^2 \leq t^2 \leq \left \lfloor \log|\A| \right \rfloor^2 \leq \dim \ell$, which holds by \propref{sB-bounds} (v).
	
	Similarly, when $b+a$ is odd, \lemref{eba-less} and \lemref{eba-more} (using $\ell \leq \log |\A|$) imply that

\begin{align*}
e_{(b,a)}(\mathcal{A}) &\leq \left ( \frac{4\sqrt{2} e}{b+a}\right )^{(b+a)} (n\ell)^{(b+a-1)/2} \log |\mathcal{A}| \cdot |\mathcal{A}|+\left ( \frac{4\sqrt{2} e}{b+a}\right )^{(b+a)}\cdot (n\ell)^{(b+a-1)/2} \ell |\mathcal{A}|\\
& = \left ( \frac{4\sqrt{2} e}{b+a}\right )^{(b+a)} \cdot |\mathcal{A}| \cdot (n\ell)^{(b+a-1)/2} \cdot (\log |\mathcal{A}| + \ell)\\
& \leq 2 \cdot \left ( \frac{4\sqrt{2} e}{b+a}\right )^{(b+a)} \cdot |\mathcal{A}| \cdot (n\ell)^{(b+a-1)/2} \cdot \log |\mathcal{A}| \quad \text{(as } \ell \leq \log |\mathcal{A}| \text{)}\\
& \leq \left ( \frac{8 e}{b+a}\right )^{(b+a)} \cdot |\mathcal{A}| \cdot (n\ell)^{(b+a-1)/2} \cdot \log |\mathcal{A}| \quad \text{(as } 2 \leq \sqrt{2}^{(b+a)} \text{)}.
\end{align*}
We claim that $\left ( \frac{8 e}{b+a}\right )^{(b+a)} \cdot |\mathcal{A}| \cdot (n\ell)^{(b+a-1)/2} \cdot \log |\mathcal{A}|$ is maximised over $\mathcal{U}$ when $b+a = 2t - 1$. Indeed, letting $k = b+a \geq 2$, we have
\begin{align*}
&\left ( \frac{8 e}{k-1}\right )^{k-1} \cdot |\mathcal{A}| \cdot (n\ell)^{(k-2)/2} \cdot \log |\mathcal{A}| \leq \left ( \frac{8 e}{k}\right )^{k} \cdot |\mathcal{A}| \cdot (n\ell)^{(k-1)/2} \cdot \log |\mathcal{A}|\\
\iff & \left(\frac{k}{k-1}\right)^{k-1} \frac{k}{8e} \leq (n\ell)^{1/2},
\end{align*}
where the last inequality holds since $(k/8)^2 \leq t^2 \leq \left \lfloor \log |\mathcal{A}|\right \rfloor^2 \leq n\ell$, by \propref{sB-bounds} (v) and $\left(\frac{k}{k-1}\right)^{k-1} \leq e$. It follows that 

\begin{align*}
e_{(b,a)}(\mathcal{A}) &\leq \left ( \frac{8 e}{2t-1}\right )^{(2t-1)} \cdot |\mathcal{A}| \cdot (n\ell)^{t-1} \cdot \log |\mathcal{A}|\\
& = \left ( \frac{4 e}{t}\right )^{2t} \cdot |\mathcal{A}| \cdot (n\ell)^{t} \cdot \log |\mathcal{A}|\cdot \left(\frac{2t}{2t-1}\right)^{(2t-1)} \cdot \frac{t}{4e}\cdot \frac{1}{n\ell}\\
& \leq \left ( \frac{4 e}{t}\right )^{2t} \cdot |\mathcal{A}| \cdot (n\ell)^{t} \cdot \log |\mathcal{A}| \cdot \frac{t}{4} \cdot \frac{1}{n\ell}\\
& \leq \left ( \frac{4 e}{t}\right )^{2t} \cdot |\mathcal{A}| \cdot (n\ell)^{t},
\end{align*}
where the last inequality follows from noting that $\frac{t \log |\mathcal{A}|}{4} \leq \left \lfloor \log |\mathcal{A}| \right \rfloor \log |\mathcal{A}| \leq n\ell$ (which follows from \propref{sB-bounds} (v)).

\end{proof}

\section{The general case for odd distances}
\begin{proof}[Proof of \thmref{odd-power-approx}.] 
	The following proof has very similar structure to the proof of \thmref{even-power-approx}, so we omit detailed calculations.
	
	Using the notation defined above, it is required to prove that if $\A \subset \{0,1\}^{\dim}$ and $t \in \mathbb{N}$ with $t \leq \log |\A|$, then
	$$
	|\mathsf{E}_{\leq 2t+1}(\A)|: = \mathsf{e}_{\leq 2t+1}(\A) \ \leq\  
	\left(\frac{16e}{2t+1}\right)^{2t+1} 
	\cdot \left(\dim \cdot \ell  \right)^t 
	\cdot |\A| \cdot \log|\A|.
	$$
Letting 
	$
	\mathcal{\calU'} = \{(b,a) \in \mathbb{Z}_{\geq 0}^2:\ b \geq a\ \text{and}\ b+a \leq 2t+1\},
	$ 
	observe that 
	\begin{equation*} \label{eqn:sum-ab-bound-odd}
	\mathsf{e}_{\leq 2t+1}(\A) 
	= \sum_{(b,a) \in \calU'} \mathsf{e}_{(b,a)}(\A).
	\end{equation*}
	We will upper bound each term in the above sum. For $(b,a) \in \calU'$, we claim that
	\begin{equation}\label{ab-2t-odd}	
	\frac{\mathsf{e}_{(b,a)}(\A)}{|\A|} 
	\leq 2\left(\frac{4\sqrt{2}e}{2t+1}\right)^{2t+1} (\dim \ell)^t \cdot \log|\A|
	\leq \left(\frac{8e}{2t+1}\right)^{2t+1} (\dim \ell)^t \cdot \log|\A|.
	\end{equation}
	Assuming that (\ref{ab-2t-odd}) holds, and using that $|\calU'| \leq 2^{2t+1}$, we have
	\begin{eqnarray*}
		\sum_{(b,a) \in \calU} \frac{\mathsf{e}_{(b,a)}(\A)}{|\A|} 
		\leq |\calU'| \cdot \left(\frac{8e}{2t+1}\right)^{2t+1} (\dim \ell)^t \cdot \log|\A|
		&\leq& \left(\frac{16e}{2t+1}\right)^{2t+1} (\dim \ell)^t\cdot \log|\A|,
	\end{eqnarray*}
	which establishes the bound in the theorem statement.
	
	We now prove (\ref{ab-2t-odd}). When $b+a$ is even, then $b+a \leq 2t$ and (\ref{ab-2t-odd}) follows from (\ref{ab-2t}). When $b+a$ is odd, then \lemref{eba-less} and \lemref{eba-more} (using $\ell \leq \log |\A|$) imply that
	\[
	\frac{\mathsf{e}_{(b,a)}(\A)}{|\A|} 
	\leq 
	\left(\frac{8e}{b+a}\right)^{b+a} \cdot (\dim\cdot \ell)^{(b+a-1)/2} \cdot \log |\A|
	\leq \left(\frac{8e}{2t+1}\right)^{2t+1} \cdot (\dim\cdot \ell)^{t} \cdot \log |\A|,
	\]
	where we use that the quantity $\left(\frac{8e}{b+a}\right)^{b+a} \cdot (\dim\cdot \ell)^{(b+a-1)/2}$ increases with $b+a$ (and is maximized over~$\calU'$ at $b+a = 2t+1$), analogous to the proof of (\ref{ba-increases}). 
\end{proof}

\section{Some open questions}
An immediate open problem is to prove exact edge isoperimetric inequalities for the graphs  we consider, i.e., to precisely determine $D(m,n,r)$ for all $(m,n,r)\in \mathbb{N}^3$. Another direction is to prove stability results for $Q_n^r$ with $r\geq 2$, generalizing prior results for sets with small edge boundary in the hypercube~\cite{ellis2018structure, keevash2018stability}.
It would also be interesting to study graphs on $[k]^{\dim}$ with $k \geq 3$ with edges induced by other metrics. For example, is it possible to prove edge isoperimetric inequalities for the families of graphs connecting pairs in $[k]^{\dim}$ with either $\ell_1$-distance at most $r$ or Hamming distance at most~$r$? Bollob\'{a}s and Leader~\cite{bollobas-leader} and Clements and Lindstr\"{o}m~\cite{clementslindstrom} have solved the respective distance one cases. 
\newpage
\bibliographystyle{plain}
\bibliography{arxiv}  

\section{Appendix}

Here we provide proof of the technical proposition, \propref{binomial-bound2}. For this we need the following tool.

\begin{proposition}
	\proplab{binomial-bound1}
	Let $f:\mathbb{R}_{\geq 0} \rightarrow \mathbb{R}$ be defined as follows 
\begin{equation*}
f(x) = \begin{cases}
\left( \frac{x}{m} \right)^m + \left( \frac{x}{m+1} \right)^{m+1} - e^{x/e} & \text{if } x \in [me,(m+1)e), \text{ for some } m \in \mathbb{N}, m \geq 1\\
1+x-e^{x/e} & \text{if } x \in [0,e)\\
\end{cases}
\end{equation*}
Then the following hold.
\begin{itemize}
\item[(1)] For $x \in [0,e)$, $f(x) \geq x/e \geq 0$.
\item[(2)] For $x \in [e,2e)$, $f(x) \geq \frac{e^2}{4}+(2-\frac{e}{4})(x-e) \geq 0$.
\item[(3)] For $m \geq 2$ and $x \in [me,(m+1)e)$, we have
\begin{equation*}
e^{x/e} - \min \left \{ \left( \frac{x}{m}\right)^m, \left( \frac{x}{m+1}\right)^{m+1} \right \} \leq \frac{1}{m} \min \left\{ \left( \frac{x}{m}\right)^m, \left( \frac{x}{m+1}\right)^{m+1} \right \},
\end{equation*}

from which it immediately follows that

\begin{equation*}
f(x) \geq \max \left\{ \left( \frac{x}{m}\right)^m, \left( \frac{x}{m+1}\right)^{m+1} \right \} - \frac{1}{m} \min \left\{ \left( \frac{x}{m}\right)^m, \left( \frac{x}{m+1}\right)^{m+1} \right \} \geq 0.
\end{equation*}
\end{itemize}
\end{proposition}

\begin{proof}
	We split our proof into parts for each of the statements.
	\paragraph{Part (1).} Suppose first that $x \in [0,e)$, so $f(x) = 1 + x - e^{x/e}$. Then $\frac{d^2f}{dx^2} = -e^{x/e -2}< 0$ and so $f$ is concave in this range. Hence, we have
	\begin{equation*}
		f(x) \geq f(0) + \frac{f(e) - f(0)}{e-0} x = \frac{x}{e},
	\end{equation*}
	as required.
	\paragraph{Part (2).} Suppose next that $x \in [e,2e)$, so that $f(x) = x + \frac{x^2}{4} - e^{x/e}$. We let
	\begin{equation*}
		g(x) = f(x) -  (\frac{e^2}{4} + (2 - \frac{e}{4})(x-e)) = (2e - \frac{e^2}{2}) + (-1 + \frac{e}{4})x + \frac{x^2}{4} - e^{x/e},
	\end{equation*}
	and note the following:
	\begin{align*}
		g'(x) &= (-1 + \frac{e}{4})+\frac{x}{2} - e^{x/e -1}\\
		g''(x) &= \frac{1}{2} - e^{x/e -2}\\
		g(e) &= g(2e) = 0.\\
	\end{align*}
	Clearly, $g''(x)$ is decreasing in $x$ and has a unique root at $x = e(2 - \ln(2))$. Therefore $g''(x) >0$ for $x \in [e, e(2-\ln(2)))$ and $g''(x) < 0$ for $x \in (e(2-\ln(2)),2e)$. We also note that $g'(e) = \frac{3e}{4} - 2 > 0$, $g'(e(2-\ln(2))) = -1 + \frac{3 - 2\ln(2)}{4}e >0$ and $g'(2e) = -1 + \frac{e}{4} < 0$.

	As $g''(x) <0$ for $x \in (e(2-\ln(2)),2e)$ and $g'(e(2 - \ln(2)))g'(2e) < 0$ we see that $g'(x) = 0$ has a unique root in $(e(2-\ln(2)),2e)$. In addition, $g''(x) >0$ for $x \in [e, e(2- \ln(2)))$ and $g'(e) g'(e(2 - \ln(2))) > 0$ so we see that $g'(x) = 0$ has no solutions in $[e,e(2-\ln(2))]$. Hence $g(x)$ has a unique maximum in $[e,2e)$, and no other stationary points. From this, and the fact that $g(e) = g(2e) = 0$ we deduce that $g(x) \geq 0$ for all $x \in [e,2e)$. This shows that
\begin{equation*}
f(x) \geq \frac{e^2}{4} + (2 - \frac{e}{4})(x-e)
\end{equation*}
for $x \in [e,2e)$, as claimed.

\paragraph{Part (3).} Suppose finally that $x \in [me,(m+1)e)$ for some $2 \leq m \in \mathbb{N}$. We now split into two cases: the case $\left ( \frac{x}{m} \right )^m \geq \left ( \frac{x}{m+1} \right )^{m+1}$, and the case $\left ( \frac{x}{m} \right )^m < \left ( \frac{x}{m+1} \right )^{m+1}$.

	\textbf{Case 1:} Suppose first that the former case holds. Then
	\begin{align*}
		e^{x/e} -  \min {\left \{ \left( \frac{x}{m} \right)^m, \left( \frac{x}{m+1}\right)^{m+1}\right \}} &= e^{x/e} - \left ( \frac{x}{m+1} \right )^{m+1}\\
		&= -\int_{t = x/e}^{m+1} { \left (\frac{x}{t} \right)^{t} (\ln \left (\frac{x}{t} \right) - 1 ) dt}\\
		& = \int_{t = x/e}^{m+1} { \left (\frac{x}{t} \right)^{t} \ln \left (\frac{t}{x/e} \right) dt}\\
		& \leq (m+1 - x/e) \max_{t \in [x/e,m+1]} \left \{  \left (\frac{x}{t} \right)^{t} \ln \left (\frac{t}{x/e} \right)  \right \}.
	\end{align*}
To bound $\max_{t \in [x/e,m+1]} \left \{  \left (\frac{x}{t} \right)^{t} \ln \left (\frac{t}{x/e} \right)  \right \}$ we  show the maximum is attained at $t = m+1$. Indeed, differentiating with respect to $t$ we get:
	\begin{align*}
		\frac{d}{dt} \left ( \left ( \frac{x}{t} \right )^t \ln \left( \frac{t}{x/e} \right)  \right ) &=  \left ( \frac{x}{t} \right )^t \left ( \frac{1}{t} -  \ln \left( \frac{t}{x/e} \right) ^ 2 \right )\\
		& \geq \left ( \frac{x}{t} \right )^t \left( \frac{1}{m+1} -   \left(\ln \left( \frac{m+1}{x/e} \right) \right)^ 2 \right ).
	\end{align*}
	It is a standard fact that for $y > 0$ we have $\frac{y-1}{y} \leq \ln(y) \leq y-1$. Noting that $\frac{m+1}{x/e} >0$, we apply this fact to see:
	\begin{equation*}
		\ln \left( \frac{m+1}{x/e} \right) \leq \frac{m+1}{x/e} - 1 = \frac{(m+1) - x/e}{ x/e} \leq e/x.
	\end{equation*}
	Hence, we have
	\begin{align*}
		\frac{d}{dt} \left ( \left ( \frac{x}{t} \right )^t \ln \left( \frac{t}{x/e} \right)  \right ) &\geq \left ( \frac{x}{t} \right )^t \left( \frac{1}{m+1} -   (e/x)^2 \right )\\
		&= \left ( \frac{x}{t} \right )^t  \left( \frac{(x/e)^2 - (m+1)}{(m+1)(x/e)^2} \right)\\
		&\geq \left ( \frac{x}{t} \right )^t \left( \frac{m^2 - m - 1}{(m+1) (x/e)^2} \right) \geq 0,
	\end{align*}
	where the final inequality holds since $m \geq 2$. Thus $\left ( \frac{x}{t} \right )^t \ln \left( \frac{t}{x/e} \right)$ is increasing on the interval $t \in [x/e,m+1]$, and attains its maximum at $t = m+1$. Therefore, we may bound the integral as follows:
	\begin{equation*}
		\int_{t = x/e}^{m+1} { \left (\frac{x}{t} \right)^{t} \ln \left (\frac{t}{x/e} \right) dt} \leq (m+1 - x/e) \left ( \frac{x}{m+1} \right )^{m+1} \ln \left( \frac{m+1}{x/e} \right) \leq \left ( \frac{x}{m+1} \right )^{m+1} \frac{1}{m} .
	\end{equation*}
The final inequality holds as $(m+1 - x/e) \leq 1$ and $\ln \left(\frac{m+1}{x/e} \right) \leq \frac{1}{m}$. The first of these is trivial, and the second can be seen as follows. We define $\epsilon \in [0,1)$ by $x = (m+\epsilon)e$, then
	\begin{equation*}
		\ln \left ( \frac{m+1}{x/e} \right ) = \ln \left ( \frac{m+1}{m+\epsilon} \right ) \leq \frac{1 - \epsilon}{m+\epsilon} \leq \frac{1}{m}.
	\end{equation*}	
	Hence, we have shown that
	\begin{equation*}
		e^{x/e} -  \min {\left \{ \left( \frac{x}{m} \right)^m, \left( \frac{x}{m+1}\right)^{m+1}\right \}} \leq \left ( \frac{x}{m+1} \right )^{m+1} \frac{1}{m},
	\end{equation*}
	i.e. that the claim holds in the former case.

	\textbf{Case 2:} Suppose secondly that the latter case holds. Then we have
	\begin{align*}
		e^{x/e} -  \min {\left \{ \left( \frac{x}{m} \right)^m, \left( \frac{x}{m+1}\right)^{m+1}\right \}} &= e^{x/e} - \left ( \frac{x}{m} \right )^{m}\\
		&= \int_{t = m}^{x/e} { \left (\frac{x}{t} \right)^{t} \ln \left (\frac{x/e}{t} \right) dt}\\
		& \leq (x/e - m) \max_{t \in [m,x/e]} \left \{  \left (\frac{x}{t} \right)^{t} \ln \left (\frac{x/e}{t} \right)  \right \}.
	\end{align*}
	To bound $\max_{t \in [m,x/e]} \left \{  \left (\frac{x}{t} \right)^{t} \ln \left (\frac{x/e}{t} \right)  \right \}$ we show that the maximum is attained at $t=m$. Differentiating with respect to $t$ we get:
	\begin{align*}
		\frac{d}{dt} \left(  \left( \frac{x}{t} \right)^t \ln \left( \frac{x/e}{t} \right) \right) &=  \left( \frac{x}{t} \right)^t \left ( \ln \left( \frac{x/e}{t} \right)^2 - \frac{1}{t} \right )\\
		& \leq \left( \frac{x}{t} \right)^t  \left (\left( \ln \left( \frac{x/e}{m} \right)\right)^2 - \frac{1}{x/e} \right ).
	\end{align*}
	Observe that
	\begin{equation*}
		\ln \left( \frac{x/e}{m} \right) \leq \frac{x/e}{m} - 1 = \frac{(x/e) - m}{m} \leq \frac{1}{m}.
	\end{equation*}
	Substituting this bound into the previous equation gives
	\begin{align*}
		\frac{d}{dt} \left(  \left( \frac{x}{t} \right)^t \ln \left( \frac{x/e}{t} \right) \right) &\leq \left( \frac{x}{t} \right)^t  \left (\left( \frac{1}{m}\right)^2 - \frac{1}{x/e} \right )\\
		&=  \left( \frac{x}{t} \right)^t  \left( \frac{x/e - m^2}{m^2 (x/e)} \right)\\
		& \leq \left( \frac{x}{t} \right)^t \left( \frac{m+1 - m^2}{m^2 (x/e)} \right) \leq 0.
	\end{align*}
	(Note that the final inequality holds as $m \geq 2$.) Hence, $ \left( \frac{x}{t} \right)^t \ln \left( \frac{x/e}{t} \right)$ is non-increasing on the interval $t \in [m,x/e]$, and so attains its maximum at $t=m$. We may bound the integral as follows:
	\begin{equation*}
		\int_{t = m}^{x/e} { \left (\frac{x}{t} \right)^{t} \ln \left (\frac{x/e}{t} \right) dt} \leq  (x/e - m)  \left (\frac{x}{m} \right)^{m} \ln \left (\frac{x/e}{m} \right) \leq  \left (\frac{x}{m} \right)^{m} \frac{1}{m}.
	\end{equation*}
	(Note that the final inequality holds as $((x/e) - m) \leq 1$ and $\ln \left (\frac{x/e}{m} \right) \leq \frac{1}{m}$. The first of these is trivial, and the second can be seen as follows. We define $\epsilon \in [0,1)$ by $x = (m+ \epsilon)e$. Then
	\begin{equation*}
		\ln\left(\frac{x/e}{m}\right) = \ln\left( \frac{m+\epsilon}{m}\right) \leq \frac{\epsilon}{m} \leq \frac{1}{m}.)
	\end{equation*}
	Hence, we have shown that
	\begin{equation*}
		e^{x/e} -  \min {\left \{ \left( \frac{x}{m} \right)^m, \left( \frac{x}{m+1}\right)^{m+1}\right \}} \leq \left ( \frac{x}{m} \right )^{m} \frac{1}{m},
	\end{equation*}
i.e. that the claim holds in the latter case. This completes the proof of the claim.
\end{proof}

We now prove \propref{binomial-bound2}.
\begin{proof}[Proof of \propref{binomial-bound2}]
Fix $m \in \mathbb{N},K \in \mathbb{R}^+$ and consider $\left( \frac{K}{m+\lambda} \right)^{m+\lambda}$. Differentiating this with respect to $\lambda$ we find:
	\begin{equation*}
		\frac{d}{d\lambda} \left( \left( \frac{K}{m+\lambda} \right)^{m+\lambda}\right) = \left( \frac{K}{m+\lambda} \right)^{m+\lambda} \left(  \ln\left ( \frac{K/e}{m+\lambda}\right ) \right).
	\end{equation*}
	The only solution to $\frac{d}{d\lambda} \left( \left( \frac{K}{m+\lambda} \right)^{m+\lambda}\right)=0$ is $\lambda = \frac{K}{e} - m$.

	If $\frac{K}{e} - m <0$, then for all $\lambda \in [0,1)$ we have $\frac{K/e}{m+ \lambda} < \frac{m}{m+\lambda} \leq 1$, so the derivative is negative, and the maximum is attained by $\left(\frac{K}{m}\right)^m$, so the claim holds in this case.

	If $\frac{K}{e} - m \geq 1$, then for all $\lambda \in [0,1)$ we have $\frac{K/e}{m+ \lambda} \geq \frac{m+1}{m+\lambda} > 1$, so the derivative is positive, and the maximum is attained by $\left(\frac{K}{m+1}\right)^{m+1}$, so the claim holds in this case also.

	Finally, suppose that $\frac{K}{e} - m \in [0,1)$. Then the maximum is at $\lambda = \frac{K}{e} - m$, but we appeal to \propref{binomial-bound1} to get
	\begin{equation*}
		\left( \frac{K}{m}\right)^{m} + \left( \frac{K}{m+1}\right)^{m+1} - \left( \frac{K}{m+\lambda}\right)^{m+\lambda} = \left( \frac{K}{m}\right)^{m} + \left( \frac{K}{m+1}\right)^{m+1} - e^{K/e} = f(K) \geq 0.
	\end{equation*}
	This leaves the case $m=0$, which we resolve similarly. First, we differentiate $(K/\lambda)^{\lambda}$ with respect to $\lambda$ to get
	\begin{equation*}
		\frac{d}{d\lambda} \left( \left( \frac{K}{\lambda} \right)^{\lambda}\right) = \left( \frac{K}{\lambda} \right)^{\lambda} \left(  \ln\left ( \frac{K/e}{\lambda}\right ) \right),
	\end{equation*}	
	and note that
\begin{itemize}
\item[(1)] the derivative has a unique root at $\lambda = K/e$,
\item[(2)] the derivative is strictly positive if $\lambda < K/e$,
\item[(3)] the derivative is strictly negative if $\lambda > K/e$.
\end{itemize}
	Consequently, if $K/e \geq 1$, then $\left(\frac{K}{\lambda}\right)^{\lambda} \leq K$ for all $\lambda \in [0,1)$, so the claim holds. If $0 < K/e < 1$ then $\left(\frac{K}{\lambda}\right)^{\lambda} \leq e ^{K/e}$, so by \propref{binomial-bound1}
	\begin{equation*}
		1 + K - \left(\frac{K}{\lambda}\right)^{\lambda} \geq 1 + K - e^{K/e} = f(K) \geq 0.
	\end{equation*}
	This completes the proof.
\end{proof}


\end{document}